\documentclass[11pt]{article} 

\usepackage{amsmath}
\usepackage{amssymb}
\usepackage{color}
\usepackage{array}
\usepackage[T1]{fontenc}
\usepackage[latin1]{inputenc}
\usepackage{datetime}
\usepackage[linkcolor=black,urlcolor=black,colorlinks=true]{hyperref}

\def\r{\rightarrow}

\usepackage{enumerate}

\newcommand{\fdem}{\hspace*{\fill}~$\Box$\par\endtrivlist\unskip}

\renewcommand{\L}{\mathbb{L}}

\newcommand{\N}{\mathbb{N}}     

\newcommand{\R}{\mathbb{R}}     
     
\newcommand{\C}{\mathbb{C}} 
 
\newcommand{\X}{\mathbb{X}}

\renewcommand{\r}{\mathop{\rightarrow}}

\newcommand{\cB}{\mbox{$\cal B$}}

\newtheorem{theo}{Theorem}
\newtheorem{pro}{Proposition}
\newenvironment{proof}[1]{\textit{Proof#1.\,}}{\fdem}
\newtheorem{lem}{Lemma}
\newtheorem{rem}{Remark}
\newtheorem{ex}{Example}
\newtheorem{cor}{Corollary}

\title{Additional material on bounds of $\ell^2$-spectral gap for discrete Markov chains with band transition matrices}

\author{Loïc HERV\'E, and James LEDOUX \footnote{INSA de Rennes, IRMAR, F-35042, France; CNRS, UMR 6625, Rennes, F-35708, France; Université Européenne de Bretagne, France. \{Loic.Herve,James.Ledoux\}@insa-rennes.fr}
}

\begin{document}

\date{{\small version du \today \ -- \currenttime}}

\maketitle
\begin{abstract}
We analyse the $\ell^2(\pi)$-convergence rate of irreducible and aperiodic Markov chains with $N$-band transition probability matrix $P$ and with invariant distribution $\pi$. This analysis is heavily based on: first the study of the essential spectral radius $r_{ess}(P_{|\ell^2(\pi)})$ of $P_{|\ell^2(\pi)}$ derived from Hennion's quasi-compactness criteria; second  the connection between  the spectral gap property (SG$_2$) of $P$ on $\ell^2(\pi)$ and the $V$-geometric ergodicity of $P$. Specifically, (SG$_2$)  is shown to hold under the condition   
\begin{equation*} 
\alpha_0 := \sum_{{m}=-N}^N \limsup_{i\r+\infty} \sqrt{P(i,i+{m})\, P^*(i+{m},i)}\ <\, 1. 
\end{equation*}  
Moreover $r_{ess}(P_{|\ell^2(\pi)}) \leq \alpha_0$. 
Simple conditions on asymptotic properties of $P$ and of its invariant probability distribution $\pi$ to ensure that $\alpha_0<1$ are given. In particular this allows us to obtain estimates of the $\ell^2(\pi)$-geometric convergence rate of random walks with bounded increments. The specific case of reversible $P$ is also addressed. Numerical bounds on the convergence rate can be provided via a truncation procedure. This is illustrated on the Metropolis-Hastings algorithm.   
\end{abstract}
\begin{center}
AMS subject classification : 60J10; 47B07

Keywords : Rate of convergence, $\ell^2$-spectral gap,  $V$-geometric ergodicity, Essential spectral radius,  Metropolis-Hastings algorithm.
\end{center}

\section{Introduction}
Let $P:=(P(i,j))_{(i,j)\in\X^2}$ be a Markov kernel on a countable state space $\X$. For the sake of simplicity we suppose that $\X:=\N$.  Throughout the paper we assume that $P$ is irreducible and aperiodic, that $P$ has a unique invariant probability measure denoted by $\pi:=(\pi(i))_{i\in\N}$ (observe that $\forall i\in\N,\ \pi(i) >0$ from irreducibility), and finally that
\begin{equation} \label{ass-voisin}
\exists i_0\in\N,\ \exists N\in\N^*,\ \forall i\geq i_0\ :\quad |i-j| > N\ \Longrightarrow\ P(i,j)=0. \tag{\textbf{AS1}}
\end{equation}
We denote by $(\ell^2(\pi),\|\cdot\|_2)$ the usual Hilbert space of sequences $(f(i))_{i\in\N}\in\C^{\N}$ such that $\|f\|_2 := [\, \sum_{i\geq0} |f(i)|^2\, \pi(i)\, ]^{1/2} < \infty$. It is well-known that $P$ 
defines a linear contraction on $\ell^2(\pi)$, and that its adjoint operator $P^*$ on $\ell^2(\pi)$ is defined by $P^*(i,j) := \pi(j)\, P(j,i)/\pi(i)$. The kernel $P$ is said to have the spectral gap property on  $\ell^2(\pi)$  at rate $\rho\in(0,1)$ if there exists some positive constants $\rho\in(0,1)$ and $C\in(0,+\infty)$ such that 
\begin{equation} \label{ineg-gap-gene}
\forall n\geq1, \forall f\in\ell^2(\pi),\quad \|P^nf - \Pi f\|_2 \leq C\, \rho^n\, \|f\|_2 \quad \text{with} \quad \Pi f := \pi(f) 1_{\N}, \tag{\textbf{SG$_2$}}
\end{equation}
where $\pi(f):=\sum_{i\geq0} f(i)\, \pi(i)$. A relevant and standard issue is to compute the value (or to find an upper bound) of 
\begin{equation} \label{def-varho-gene}
\varrho_2 := \inf\{\rho\in(0,1) : \text{(\ref{ineg-gap-gene}) holds true}\}.
\end{equation}

In this work we use the quasi-compactness criteria of \cite{Hen93} to study (\ref{ineg-gap-gene}) and to estimate $\varrho_2$. In Section~\ref{sec-bounded-tp} it is proved that (\ref{ineg-gap-gene}) holds when 
\begin{equation} \label{ass-beta}
\alpha_0 := \sum_{{m}=-N}^N \limsup_{i\r+\infty} \sqrt{P(i,i+{m})\, P^*(i+{m},i)}\ <\, 1. \tag{\textbf{AS2}}
\end{equation}
Moreover $r_{ess}(P_{|\ell^2(\pi)}) \leq \alpha_0$. The main argument to obtain this result is the Doeblin-Fortet inequality in Lemma~\ref{lem-D-F-gene}. We refer to \cite{Hen93} for the definition of the essential spectral radius $r_{ess}(T)$ (related to quasi-compactness) of a bounded linear operator $T$ on a Banach space.  In Section~\ref{sec-stab-expo}, under the following assumptions 
\begin{gather} 
 \quad \forall {m}=-N,\ldots,N,\quad P(i,i+{m}) \xrightarrow[i\r +\infty]{} a_{m}\in[0,1]. \label{cond-lim-intro} \tag{\textbf{AS3}}\\
\frac{\pi(i+1)}{\pi(i)} \xrightarrow[i\r +\infty]{} \tau \in[0,1) \label{pi-tail} \tag{\textbf{AS4}}\\
\sum_{k=-N}^{N} k\, a_{k}\, < 0, \label{MoySautBorne} \tag{\textbf{NERI}}
\end{gather}
we establish that (\ref{ass-beta}) holds (hence (\ref{ineg-gap-gene})) and that $\alpha_0$ can be explicitly computed in function of $\tau$ and the $a_{m}$'s. 
Observe that (\ref{MoySautBorne}) means that the expectation of the asymptotic random increments is negative. Moreover, using the inequality  $r_{ess}(P_{|\ell^2(\pi)}) \leq \alpha_0$, Property~(\ref{ineg-gap-gene}) is proved to be connected to the so-called $V$-geometric ergodicity of $P$ for $V:= (\pi(n)^{-1/2})_{n\in\N}$, which corresponds to the spectral gap property on the usual weighted-supremum space $\cB_V$ associated with $V$. In particular, denoting the minimal
$V$-geometrical ergodic rate by $\varrho_V$, it is proved that, either $\varrho_2$ and $\varrho_V$ are both less than $\alpha_0$, or $\varrho_2=\varrho_V$. As a result, an accurate bound of $\varrho_2$ is obtained for random walks (RW) with i.d.~bounded increments using the results of \cite{HerLedJAP14}. In the reversible case (Section~\ref{sec-reversible}) the previous results hold under Assumptions~(\ref{cond-lim-intro}) and (\ref{pi-tail}) provided that $a_m \neq a_{-m}$ for at least one $m$. A first illustration to Birth-and-Death Markov chains (BDMC) is proposed in Subsection~\ref{subsec-ex-bdmc}. The reversible case naturally contains the Markov kernels associated with the Metropolis-Hastings (M-H) Algorithm. In Subsection~\ref{sec-MCMC} we observe that, if the target distribution $\pi$ and the proposal kernel $Q:=(Q(i,j))_{(i,j)\in\N^2}$ satisfy (\ref{ass-voisin}), (\ref{cond-lim-intro}) and (\ref{pi-tail}), then so is the associated reversible M-H kernel $P$, which then satisfies (\ref{ineg-gap-gene}). 

Estimating $\varrho_2$ is a difficult but relevant issue. This question is investigated in Section~\ref{cas-gap-connu} where an accurate estimation of $\varrho_2$ is obtained by using the above mentioned link between $\varrho_2$ and $\varrho_V$ and by applying the truncation procedure in \cite{HerLed14a}. 
Numerical applications to discrete MCMC are presented at the end of Section~\ref{cas-gap-connu}. Bounding $\varrho_2$ in the reversible case is of special interest  since (\ref{ineg-gap-gene}) holds in this case with $C=1$ and $\rho=\varrho_2$.

The spectral gap property for Markov processes has been widely investigated in the discrete and continuous-time cases (e.g.~see \cite{Ros71} for discrete-time, \cite{Che04} for continuous-time, and \cite{ConGui13} for dynamical systems). We point out that there exist different definitions of the spectral gap property according that we are concerned with discrete or continuous-time case. A simple and concise presentation about this difference is proposed in \cite{Yue00,MaoSon13}. The focus of our paper is on the discrete time case. In the reversible case, the equivalence between the geometrical ergodicity and (\ref{ineg-gap-gene}) is proved in \cite{RobRos97} and  Inequality $\varrho_2 \leq \varrho_{V}$ is obtained in \cite[Th.6.1.]{Bax05}.  This equivalence fails in the non-reversible case (see \cite{KonMey12}). The link between $\varrho_2$ and $\varrho_{V}$ stated in our Proposition~\ref{pro-RW-SG} is obtained with no reversibility condition. The works \cite{StaWub11,Wu12} provide formulae for $\varrho_2$ in terms of isoperimetric constants which are related to $P$ in reversible case and to $P$ and $P^*$ in non-reversible case. However, to the best of our knowledge, no explicit value (or upper bounds) of $\varrho_2$ can be derived from these formulae for discrete Markov chains with band transition matrices. For instance (\ref{ineg-gap-gene}) is proved to hold in \cite{Wu12} for RW with i.d.~bounded increments satisfying (\ref{MoySautBorne}) and a weak reversibility condition, but no explicit bounds for $\varrho_2$ are derived from isoperimetric constants. For such RWs, our method gives the exact value of $\varrho_2$ with no reversibility assumption (see Examples~\ref{ex-rw-gene} and \ref{ex-g-2-d-1}). Concerning BDMCs, recall that the decay parameter of $P$, which equals to $\varrho_2$ for these models (see \cite{DooSch95}), is only known for specific instances of BDMC (see Remark~\ref{l2-spectral-gap} for details). 
In the context of discrete MCMC, no satisfactory bound for $\varrho_2$ was known to the best of our knowledge, except for special instances as the simulation of a geometric distribution corresponding to a simple BDMC (see \cite[Ex.~2]{MenTwe96}). The bounds for $\varrho_2$ obtained in Section~\ref{cas-gap-connu} for discrete MCMC via truncation procedure applies to any target distribution $\pi$ satisfying (\ref{pi-tail}) when  the proposal kernel $Q$ satisfies (\ref{ass-voisin}) and (\ref{cond-lim-intro}). The accuracy of our estimation in Section~\ref{cas-gap-connu} depends on the order $k$ of the used truncated finite matrix $P_k$ (see Tables~\ref{Table_tau=3/5} and \ref{Table_Poisson}). Our explicit bound $r_{ess}(P_{|\ell^2(\pi)}) \leq \alpha_0$ in Theorem~\ref{theo-spec-gap-gene} for discrete Markov chains with band transition matrices is the preliminary key results in this work. Recall that $r_{ess}(P_{|\ell^2(\pi)})$ is a natural lower bound of $\varrho_2$ (see \cite[Prop.~2.1]{HerLedJAP14} with $\ell^2(\pi)$ in place of $\cB_V$).  The essential spectral radius of Markov operators on a $\L^2$-type space is investigated for discrete-time Markov chains with general state space in \cite{Wu04} (see also \cite{GonWu06}), but no explicit bound for $r_{ess}(P_{|\ell^2(\pi)})$ can be derived a priori from these theoretical results for discrete Markov chains with band transition matrices, except Inequality $r_{ess}(P_{|\ell^2(\pi)}) \leq r_{ess}(P_{|{\cal B}_V})$ in the reversible case (see \cite[Th.~5.5.]{Wu04}). 
Finally recall that, for any Markov chain $(X_n)_{n\in\N}$ with transition kernel $P$ satisfying {(\ref{ineg-gap-gene})}, the Berry-Esseen theorem and the first-order Edgeworth expansion apply to additive functional of $(X_n)_{n\in\N}$ under the expected third-order moment condition, see \cite{FerHerLed10}.

\section{(\ref{ineg-gap-gene}) under Assumption~(\ref{ass-voisin}) on $P$} \label{sec-bounded-tp}
%
\begin{theo} \label{theo-spec-gap-gene}
If Condition~\emph{(\ref{ass-beta})} holds, then $P$ satisfies \emph{(\ref{ineg-gap-gene})}. Moreover $r_{ess}(P_{|\ell^2(\pi)}) \leq \alpha_0$. 
\end{theo}
%
\begin{proof}{} 
Let $\ell^1(\pi)$ denote the usual Banach space of sequences $(f(i))_{i\in\N}\in\C^{\N}$ satisfying the following condition: $\|f\|_1 := \sum_{i\geq0} |f(i)|\, \pi(i)\ < \infty$. 
%
\begin{lem} \label{lem-id-comp}
The identity map is compact from $\ell^2(\pi)$ into $\ell^1(\pi)$. 
\end{lem}
%
%
\begin{lem} \label{lem-D-F-gene}
For any $\alpha>\alpha_0$, there exists a positive constant $L\equiv L(\alpha)$ such that 
$$\forall f\in\ell^2(\pi),\quad \|Pf\|_2 \leq \alpha\, \|f\|_2 + L\|f\|_1.$$
\end{lem}
It follows from these lemmas and from \cite{Hen93} that $P$ is quasi-compact on $\ell^2(\pi)$ with $r_{ess}(P_{|\ell^2(\pi)}) \leq \alpha$. Since $\alpha$ can be chosen arbitrarily close to $\alpha_0$, this gives $r_{ess}(P_{|\ell^2(\pi)}) \leq \alpha_0$. Then (\ref{ineg-gap-gene}) is deduced from aperiodicity and irreducibility assumptions.
\end{proof}
Lemma~\ref{lem-id-comp} follows from the Cantor diagonal procedure. 

\noindent\begin{proof}{ of Lemma~\ref{lem-D-F-gene}}
Under Assumption~(\ref{ass-voisin}) we define 
\begin{equation} \label{beta-k}
\forall i\geq i_0,\ \forall {m}=-N,\ldots,N,\quad \beta_{m}(i) := \sqrt{P(i,i+{m})\, P^*(i+{m},i)}.
\end{equation} 
Let $\alpha>\alpha_0$, with $\alpha_0$ given in (\ref{ass-beta}). Fix $\ell\equiv \ell(\alpha) \geq i_0$ such that 
$\sum_{{m}=-N}^{N} \,  \sup_{i\geq \ell} \beta_{m}(i) \leq \alpha$. For $f\in\ell^2(\pi)$ we have from Minkowski's inequality and the band structure of $P$ for $i\geq \ell$ 
\begin{eqnarray*}
\|Pf\|_2 &\leq& \bigg[\sum_{i<\ell}\big|(Pf)(i)\big|^2\pi(i)\, \bigg]^{1/2} 
+\ \bigg[\sum_{i\geq \ell} \bigg|\sum_{{m}=-N}^{N} P(i,i+{m})\, f(i+{m})\bigg|^2\pi(i)\, \bigg]^{1/2} \nonumber \\ 
&\leq& C_\ell\, \sum_{i<\ell} |(Pf)(i)|\, \pi(i)\,  + \bigg[\sum_{i\geq \ell} \bigg|\sum_{{m}=-N}^{N} P(i,i+{m})\, f(i+{m})\bigg|^2\pi(i)\, \bigg]^{1/2} 
\end{eqnarray*}
where $C_\ell>0$ is derived from equivalent norms on the space $\C^{\ell}$. Note that $\sum_{i<\ell}|(Pf)(i)|\, \pi(i) \leq \|Pf\|_1 \leq \|f\|_1$ so that setting $L:= C_{\ell}$
\begin{eqnarray} \label{ineg-DF-inter1}
\|Pf\|_2 &\leq& L \|f\|_1 +  \bigg[\sum_{i\geq \ell} \bigg|\sum_{{m}=-N}^{N} P(i,i+{m})\, f(i+{m})\bigg|^2\pi(i)\, \bigg]^{1/2}. 
\end{eqnarray}


It remains to obtain the expected control of the second terms in the right hand side of (\ref{ineg-DF-inter1}). For ${m}=-N,\ldots,N$, let us define $F_{m}=(F_{m}(i))_{i\in\N}\in\ell^2(\pi)$ by
$$F_{m}(i):= \left \{
    \begin{array}{ll}
      \quad \quad \quad\quad 0 \quad \quad \quad \quad \quad \ \  \text{if } i<\ell \\
      P(i,i+{m})\, f(i+{m}) \quad\quad \text{if } i \geq \ell. 
    \end{array}
    \right.
$$
Then 
\begin{eqnarray*}
\lefteqn{\bigg[\sum_{i\geq \ell} \bigg|\sum_{{m}=-N}^{N} P(i,i+{m})\, f(i+{m})\bigg|^2\pi(i)\, \bigg]^{1/2}= \big\|\sum_{{m}=-N}^{N} F_{m}\|_2}\\
 &\leq &   \sum_{{m}=-N}^{N} \| F_{m}\|_2 =\sum_{{m}=-N}^{N}  \bigg[\sum_{i\geq \ell} P(i,i+{m})^2\, |f(i+{m})|^2\pi(i)\, \bigg]^{1/2} \\
&& =  \sum_{{m}=-N}^{N}  \bigg[\sum_{i\geq \ell} P(i,i+{m})\frac{\pi(i)\, P(i,i+{m})}{\pi_{i+{m}}}\, |f(i+{m})|^2\pi_{i+{m}}\, \bigg]^{1/2} \ \text{(from the definition of $P^*$)}\\ 
& \leq &  \sum_{{m}=-N}^{N}  \big(\sup_{i\geq \ell} \beta_{m}(i)\big) \bigg[\sum_{i\geq \ell}  |f(i+{m})|^2\pi_{i+{m}}\, \bigg]^{1/2} \qquad \text{(from (\ref{beta-k}))}\\ 
\\
& \leq & \bigg(\sum_{{m}=-N}^{N}  \sup_{i\geq \ell} \beta_{m}(i)\bigg)\, \|f\|_2.
\end{eqnarray*}
The statement in Lemma~\ref{lem-D-F-gene} can be deduced from the previous inequality and from (\ref{ineg-DF-inter1}). \end{proof}

\section{(\ref{ineg-gap-gene}) and geometric ergodicity. Application to RWs with i.d.~bounded increments} \label{sec-stab-expo}
We specify Theorem~\ref{theo-spec-gap-gene} in terms of $V-$geometric ergodicity for $V := ({\pi(n)}^{-1/2})_{n\in\N}$. Let $(\cB_{V},\|\cdot\|_{V})$ denote the weighted-supremum space of sequences $(g(n))_{n\in\N}\in\C^\N$ such that $\|g\|_{V} := \sup_{n\in\N} V(n)^{-1}\, |g(n)| < \infty$.
Recall that $P$ is said to be $V$-geometrically ergodic if $P$ satisfies the spectral gap property on $\cB_V$, namely: there exists $C\in(0,+\infty)$ and $\rho\in (0,1)$ such that  
\begin{equation} \label{ineg-gap-V} 
\forall n\geq1, \forall f\in\cB_V,\quad \|P^nf - \Pi f\|_V \leq C\, \rho^n\, \|f\|_V. \tag{\textbf{SG$_V$}} 
\end{equation}
When this property holds, we define 
\begin{equation} \label{def-rho-V} 
\varrho_V := \inf\{\rho\in(0,1) : \text{(\ref{ineg-gap-V}) holds true}\}.
\end{equation} 


%
\begin{rem} \label{rem_fle-alpha}
Under Assumptions~\emph{(\ref{cond-lim-intro})} and \emph{(\ref{pi-tail})}, we have 
\begin{equation} \label{fle-alpha}
\alpha_0 := \sum_{{m}=-N}^N \limsup_{i\r+\infty} \sqrt{P(i,i+{m})\, P^*(i+{m},i)} 
 = \begin{cases}
 \text{$\displaystyle\sum_{{m}=-N}^{N} a_{m}\, \tau^{-{m/2}}$} & \text{if }\  \tau\in(0,1) \\[0.12cm]
\quad  a_0 & \text{if }\  \tau=0,  \\
\end{cases}
\end{equation}
Moreover, if $\tau=0$ in \emph{(\ref{pi-tail})}, then $a_{m}=0$ for every ${m}=1,\ldots,N$. 

Indeed, if \emph{(\ref{pi-tail})} holds with $\tau\in(0,1)$, then the claimed formula follows from the definition of $P^*(\cdot,\cdot)$. If $\tau=0$, 
then $a_{m}=0$ for every $m >0$ 
since the invariance of $\pi$ gives 
\begin{equation} \label{sum-1-inter}
\sum_{m=-N}^N P(i+m,i)\frac{\pi(i+m)}{\pi(i)} = 1,
\end{equation} 
so that the sequence $\big(P(i+m,i) \pi(i+m)/\pi(i)\big)_i$ must be bounded for each $m<0$. Now observe that we have for $m<0$ 
$$\sqrt{P(i,i+{m})\, P^*(i+{m},i)} = P(i,i+{m})\sqrt{\frac{\pi(i)}{\pi(i+m)}} \longrightarrow 0 \quad \text{when  $i\r+\infty$.}$$
Next, setting $\ell = i+m$, we obtain for $m>0$ 
\begin{eqnarray*}
\sqrt{P(i,i+{m})\, P^*(i+{m},i)} &=& P(\ell-m,\ell)\sqrt{\frac{\pi(\ell-m)}{\pi(\ell)}} \\
&=& P(\ell-m,\ell)\frac{\pi(\ell-m)}{\pi(\ell)}\sqrt{\frac{\pi(\ell)}{\pi(\ell-m)}} 
\ \ \  \longrightarrow 0 \quad \text{when  $i\r+\infty$}
\end{eqnarray*}
since we know that $\big(P(\ell-m,\ell)\, \pi(\ell-m)/\pi(\ell)\big)_\ell$ is bounded. Hence $\alpha_0=a_0$. 
\end{rem}
%
\begin{pro}  \label{pro-RW-SG} 
If $P$ and $\pi$ satisfy Assumptions~\emph{(\ref{cond-lim-intro})}, \emph{(\ref{pi-tail})} and \emph{(\ref{MoySautBorne})}, then $P$ satisfies~\emph{(\ref{ass-beta})} (and~$\alpha_0<1$ with $\alpha_0$ given in (\ref{fle-alpha})). Moreover $P$ satisfies both \emph{(\ref{ineg-gap-gene})} and \emph{(\ref{ineg-gap-V})}, we have $\max(r_{ess}(P_{|{\cal B}_V}),r_{ess}(P_{|\ell^2(\pi)})) \leq \alpha_0$, and the following assertions hold: 
\begin{enumerate}[(a)]
  \item if $\varrho_{V} \leq \alpha_0$, 
	then $\varrho_2 \leq \alpha_0$; 
	\item if $\varrho_{V} > \alpha_0$, 
	then $\varrho_2 = \varrho_{V}$.
\end{enumerate}
\end{pro}
%
\begin{proof}{}
If $\tau=0$ in (\ref{pi-tail}), then $\alpha_0=a_0<1$ from (\ref{fle-alpha}) and (\ref{MoySautBorne}). Now assume that (\ref{pi-tail}) holds with $\tau\in(0,1)$. Then $\alpha_0=\sum_{m=-N}^{N} a_{m}\, \tau^{-{m/2}} = \psi(\sqrt{\tau})$, where:  
$\forall t>0,\ \psi(t) := \sum_{k=-N}^{N} a_{k}\, t^{-k}$. Moreover it easily follows from the invariance of $\pi$ that $\psi(\tau)=1$. Inequality $\alpha_0 =\psi(\sqrt{\tau}) < 1$ then follows from the following assertions: $\forall t \in(\tau,1),\ \psi(t)<1$ and $\forall t \in (0,\tau) \cup (1,+\infty),\ \psi(t) > 1$. To prove these properties, note that $\psi(\tau)=\psi(1)=1$ and that $\psi$ is convex on $(0,+\infty)$ since the second derivative of $\psi$ is positive on $(0,+\infty)$. Moreover we have $\lim_{t\r+\infty} \psi(t)= +\infty$ since $a_k>0$ for some $k<0$ (use $\psi(\tau)=\psi(1)=1$ and $\tau\in(0,1)$). Similarly, $\lim_{t\r 0^{+}} \psi(t) =  +\infty$ since $a_k>0$ for some $k>0$. This gives the desired properties on $\psi$ since $\psi'(1) > 0$ from~(\ref{MoySautBorne}).    

(\ref{ineg-gap-gene}) and $r_{ess}(P_{|\ell^2(\pi)}) \leq \alpha_0$ follow from Theorem~\ref{theo-spec-gap-gene}. Next (\ref{ineg-gap-V}) is deduced from the well-known link (see \cite{MeyTwe93}) between geometric ergodicity and the following drift inequality: 
\begin{equation} \label{drift-V}
\forall\alpha\in(\alpha_0,1),\ \exists L\equiv L_\alpha>0, \quad PV \leq \alpha V + L\, 1_{\N}. 
\end{equation}
This inequality holds from 
$$\frac{(PV)(i)}{V(i)} 
= \sum_{m=-N}^N P(i,i+m)\left(\frac{\pi(i)}{\pi(i+m)}\right)^{\frac{1}{2}} \xrightarrow[i\r+\infty]{}  \alpha_0.
$$
This gives (\ref{drift-V}), from which (\ref{ineg-gap-V}) is derived using aperiodicity and irreducibility. It also follows from  (\ref{drift-V}) that $r_{ess}(P_{|{\cal B}_V}) \leq \alpha$ (see \cite[Prop.~3.1]{HerLedJAP14}). Thus $r_{ess}(P_{|{\cal B}_V}) \leq \alpha_0$. 

Now we prove $(a)$ and $(b)$ using the spectral properties of \cite[Prop.~2.1]{HerLedJAP14} of both $P_{|\ell^2(\pi)}$ and $P_{|{\cal B}_V}$ (due to quasi-compactness). We will also use the following obvious inclusion: $\ell^2(\pi) \subset \cB_{V}$. In particular  every eigenvalue of $P_{|\ell^2(\pi)}$ is also an eigenvalue for $P_{|{\cal B}_V}$. 
First assume that $\varrho_{V} \leq \alpha_0$. Then there is no eigenvalue for $P_{|{\cal B}_V}$ in the annulus $\Gamma:=\{\lambda\in\C : \alpha_0 < |\lambda| < 1\}$ since $r_{ess}(P_{|{\cal B}_V}) \leq \alpha_0$. From $\ell^2(\pi) \subset \cB_{V}$ it follows that there is also no eigenvalue for $P_{|\ell^2(\pi)}$ in this  annulus. Hence $\varrho_2 \leq \alpha_0$ since $r_{ess}(P_{|\ell^2(\pi)}) \leq \alpha_0$. 
Second assume that $\varrho_{V} > \alpha_0$. Then $P_{|{\cal B}_V}$ admits an eigenvalue $\lambda\in\C$ such that $|\lambda| = \varrho_{V}$. Let $f\in\cB_{V}$, $f\neq 0$, such that $Pf = \lambda f$. We know from \cite[Prop.~2.2]{HerLedJAP14} that there exists some $\beta\equiv\beta_\lambda\in(0,1)$ such that $|f(n)| = \text{O}(V(n)^{\beta}) = \text{O}(\pi(n)^{-\beta/2})$, so that $|f(n)|^2\pi(n) = \text{O}(\pi(n)^{(1-\beta)})$, thus $f\in \ell^2(\pi)$ from (\ref{pi-tail}). We have proved that $\varrho_2 \geq \varrho_{V}$. Finally the converse inequality is true since every eigenvalue of $P_{|\ell^2(\pi)}$ is an eigenvalue for $P_{|{\cal B}_V}$. Thus $\varrho_2 = \varrho_{V}$.  
\end{proof}
%
\begin{ex} [RWs with i.d.~bounded increments] \label{ex-rw-gene}
Let $P$ be defined as follows. There exist some positive integers $c,g,d\in\N^*$ such that 
\begin{subequations}
\begin{gather}
\forall i\in\{0,\ldots,g-1\},\quad \sum_{j = 0}^{c} P(i,j)=1; \label{cond-bound-prob} \\
\forall i\ge g, \forall j\in\N, \quad P(i,j) = \begin{cases}
 a_{j-i} & \text{if }\  i-g\leq j \leq i+d \\
 0 & \text{otherwise.}  \\
\end{cases} \label{Def_HRW-ter} \\
(a_{-g},\ldots,a_d)\in[0,1]^{g+d+1} : a_{-g}>0, \ a_d>0, \ \sum_{k=-g}^{d} a_k=1. 
\label{Def_HRW-bis} 
\end{gather} 
\end{subequations}
%


\noindent We assume that $P$ is aperiodic and irreducible, and that Assumtion~\emph{(\ref{MoySautBorne})} holds, that is: $\sum_{k=-g}^{d} k\, a_{k}\, < 0$. Then $P$ admits a unique invariant distribution $\pi$, and the conclusions of Proposition~\ref{pro-RW-SG} hold. Moreover it can be derived from standard results of linear difference equation that $\pi(n) \sim c\, \tau^n$ when $n\r+\infty$, with $\tau\in(0,1)$ defined by $\psi(\tau)  = 1$, where $\psi(t) := \sum_{k=-N}^{N} a_{k}\, t^{-k}$. Thus, if $\gamma := \tau^{-1/2}$, then $\cB_V=\{(g(n))_{n\in\N}\in\C^\N,\ \sup_{n\in\N} \gamma^{-n}\, |g(n)| < \infty\}$. Then we know from \cite[Prop.~3.2]{HerLedJAP14} that $r_{ess}(P_{|{\cal B}_V}) =\alpha_0$ with $\alpha_0$ given in (\ref{fle-alpha}), and that $\varrho_{V}$ can be computed from an algebraic polynomial elimination. More precisely, the procedure in \cite{HerLedJAP14} developed for a special value $\hat\gamma$ can be applied for $\gamma := \tau^{-1/2}$ by considering $\Gamma:=\{\lambda\in\C : \psi(\sqrt{\tau}) < |\lambda| < 1\}$. When Assertion~$(b)$ of Proposition~\ref{pro-RW-SG} applies, we obtain the exact value of $\varrho_2$  (see Example~\ref{ex-g-2-d-1}). Property~\emph{(\ref{ineg-gap-gene})} is proved in \cite[Th.~2]{Wu12} under an extra weak reversibility assumption (with no explicit bound on $\varrho_2$). However, except in case $g=d=1$ where reversibility is automatic, a RW with i.d.~bounded increments is not reversible or even weak reversible in general. Note that no reversibility condition is required in Proposition~\ref{pro-RW-SG}. 
\end{ex}
%
\begin{ex} [Numerical examples in case $g=2$ and $d=1$] \label{ex-g-2-d-1}
Let $P$ be defined by 
\begin{gather}
P(0,0) = a \in(0,1),\quad P(0,1) = 1-a,\quad P(1,0) = b\in(0,1), \quad P(1,2) = 1-b \label{bound-deux-vois} \\
\forall n\geq 2,\ P(n,n-2) = 1/2,\ P(n,n-1) = 1/3,\ P(n,n) = 0, \ P(n,n+1) = 1/6. \label{trans-deux-vois}
\end{gather}
The form of boundary probabilities in (\ref{bound-deux-vois}) and the special values in (\ref{trans-deux-vois}) are chosen for convenience. Other (finitely many) boundary probabilities in (\ref{bound-deux-vois}) and other values in (\ref{trans-deux-vois})  could be considered provided that $P$ is irreducible and aperiodic and that $(a_{-2},a_{-1},a_0,a_1)$ satisfies $a_{-2}, a_1 >0$ and \emph{(\ref{MoySautBorne})} i.e. $a_1 < 2a_{-2} + a_{-1}$. Here the fonction $\psi$ is given by: $\psi(t) := t^2/2  + t/3 + 1/6t = 1+ (t-1)(t^2-5t/3-1/3)/2t$. 
Then function $\psi(\cdot)-1$ has a unique zero over $(0,1)$  which is $\tau = (\sqrt{37}-5)/6 \approx 0.1805$ and $\alpha_0 = \psi(\sqrt{\tau}) \approx 0.6242$. Let $\gamma := 1/\sqrt{\tau} \approx 2.3540$ and $V :=(\gamma^n)_{n\in\N}$. Using the procedure from \cite{HerLedJAP14} and Proposition~\ref{pro-RW-SG},  we give in Table~\ref{Table} the values of $\alpha_0$, $\varrho_{V}$ and $\varrho_2$ for this instance.
\begin{table}[h]
    \renewcommand{\arraystretch}{1.2}
\centering
\begin{tabular}{c||c|c||c|c||c|c} \hline
$(a,b)$ &  $\alpha_0$ & $\rho_{V}$  & $\varrho_2$  \\\hline
$(1/2,1/2)$ &  0.624 & 0.624 & $\leq 0.624$\\\hline
$(1/10,1/10)$ &  0.624 & 0.688 & 0.688 \\\hline
$(1/50,1/50)$ & 0.624 & 0.757& 0.757\\\hline
\end{tabular}
\caption{Convergence rate on $\ell^2(\pi)$ for different 
boundary transition probabilities $(a,b)$} 
\label{Table}
\end{table}
\end{ex}
%
\begin{rem}
If \emph{(\ref{pi-tail})} in Proposition~\ref{pro-RW-SG} is reinforced by the condition $\pi(n) \sim c\, \tau^n$ when $n\r+\infty$ with $\tau\in(0,1)$ (e.g.~see Example~\ref{ex-rw-gene}), then let us consider $\cB_V=\{(g(n))_{n\in\N} \in \C^\N,\ \sup_{n\in\N} \gamma^{-n} \, |g(n)| < \infty\}$ with $\gamma := \tau^{-1/2}$. Then we deduce from \cite[Prop.~3.2]{HerLedJAP14} that $r_{ess}(P_{|{\cal B}_V}) =\alpha_0$ with $\alpha_0$ given in (\ref{fle-alpha}), so that $\varrho_{V} \leq \alpha_0$ implies that $\varrho_{V} = \alpha_0$ since $\varrho_{V} \geq r_{ess}(P_{|{\cal B}_V})$. Then it follows from Proposition~\ref{pro-RW-SG} that $\varrho_2 \leq \varrho_{V}$ and that this inequality is an equality when $\varrho_{V}>\alpha_0$. The passage from \emph{(\ref{ineg-gap-V})} to \emph{(\ref{ineg-gap-gene})} and the inequality $\varrho_2 \leq \varrho_{V}$ was established in \cite{RobRos97,Bax05} for general reversible $V$-geometrically ergodic Markov kernels. Again note that no reversibility condition is assumed in Proposition~\ref{pro-RW-SG}. 
\end{rem}


\section{Applications to the reversible case} \label{sec-reversible}
The reversible case corresponds to the condition $P=P^*$ (i.e.~$P$ is self-adjoint in $\ell^2(\pi)$), namely: $\forall (i,j)\in\N^2,\ \pi(i) \, P(i,j) = \pi(j)  \, P(j,i)$ (detailed balance condition). Then (\ref{ineg-gap-gene}) is equivalent to the condition $\varrho_2=\|P-\Pi\|_2 < 1$, where $\|\cdot\|_2$ denotes here the operator norm on $\ell^2(\pi)$. Thus, when  (\ref{ineg-gap-gene}) holds in the reversible case, we have $C=1$ and $\rho=\varrho_2$, that is  
\begin{equation} \label{ineg-gap-intro-rho2}
\forall n\geq 1,\ \forall f\in\ell^2(\pi),\quad \|P^nf - \pi(f){\bf 1}\|_2 \leq \, {\varrho_2}^n\, \|f\|_2. 
\end{equation} 
%
\begin{cor} \label{rem-surprise} 
If $P$ is reversible, then: 
\leftmargini 1.2em
\begin{enumerate}
  \item $P$ satisfies \emph{(\ref{ineg-gap-gene})} and $r_{ess}(P_{|\ell^2(\pi)}) \leq \alpha_0$ under \emph{(\ref{ass-beta})}, with:  
\begin{equation*} 
\alpha_0 :=  \sum_{{m}=-N}^N \left(\limsup_{i\r+\infty} \sqrt{P(i,i+{m})\, P(i+{m},i)}\right)\  <1.
\end{equation*}

	\item If Condition~\emph{(\ref{cond-lim-intro})} holds true, then 
\begin{equation} \label{cond-lim-rev-bis}
\alpha_0 = 1 - \sum_{m=1}^N\big(\sqrt{a_m} - \sqrt{a_{-m}}\big)^2.
\end{equation}
Consequently, 
if $a_m \neq a_{-m}$ for at least one $m\in\{1,\ldots,N\}$, then $P$ satisfies \emph{(\ref{ass-beta})}. 
\item If $P$ satisfies \emph{(\ref{cond-lim-intro})} and if $\pi$ satisfies \emph{(\ref{pi-tail})} with $\tau\in(0,1)$, then $a_m \neq a_{-m}$ provided that $a_{m} \neq 0$. Consequently, if $a_m \neq 0$ for some $|m| \in \{1,\ldots,N\}$,  
then the conclusions of Proposition~\ref{pro-RW-SG} hold with $\alpha_0$ given in~(\ref{cond-lim-rev-bis}). 
\item If $P$ satisfies \emph{(\ref{cond-lim-intro})} with $a_0<1$ and if $\pi$ satisfies \emph{(\ref{pi-tail})} with $\tau=0$, then the conclusions of Proposition~\ref{pro-RW-SG} hold. 
\end{enumerate}
\end{cor}

- Dans le corollaire 1, je mettrai : Consequently, if $a_m \neq 0$ for
$|m| \in \{1,\ldots,N\}$ car vaut aussi pour les $m<0$

Keep in mind that all our results are stated for positive recurrent Markov kernels. For instance, for Markov chain associated with $P(i,i-1):=p,  P(i,i):=r, P(i,i+1):=q$ where $p+r+q=1$, Formula~(\ref{cond-lim-rev-bis}) is $\alpha_0 = 1 - (\sqrt{q} - \sqrt{p})^2$, but the existence of $\pi$ is only guaranteed  when $p>q$.  

\noindent \begin{proof}{}
The first statement follows from Theorem~\ref{theo-spec-gap-gene} and reversibility. Next (\ref{cond-lim-intro}) gives $\alpha_0 =  \sum_{{m}=-N}^N \sqrt{a_{m}\, a_{-m}}$, hence Assertion~2.~since $\sum_{m=-N}^N a_m=1$. If moreover (\ref{pi-tail}) holds with $\tau\in(0,1)$, then $a_m \neq a_{-m}$ for every $m\in\{1,\ldots,N\}$ since $\tau^m\, a_{-m} = a_m$ from the balance condition. Thus, under (\ref{cond-lim-intro}) and (\ref{pi-tail}) with $\tau\in(0,1)$, we obtain from Assertion~2.~that $\alpha_0 <1$. Moreover, since the real numbers $\alpha_0$ given in~(\ref{cond-lim-rev-bis}) and in (\ref{fle-alpha}) are equal, all the spectral properties obtained in Proposition~\ref{pro-RW-SG} remain valid. Idem for Assertion~4.~from Remark~\ref{rem_fle-alpha}. 
\end{proof}

\subsection{Birth-and-Death Markov chains (BDMC)} \label{subsec-ex-bdmc}

The transition kernel $P:=(P(i,j))_{(i,j)\in\N^2}$ of a Birth-and-Death Markov chains is defined by
\begin{equation} \label{matrix-BDMC}
	P := \begin{pmatrix}
	r_0 & q_0 &0 & \cdots & \cdots\\
	p_1 & r_1 & q_1 & \ddots &  \\
	0 & 	p_2 & r_2 & q_2 & \ddots &  \\
	\vdots & \ddots & \ddots & \ddots & \ddots  
	\end{pmatrix}.
\end{equation}
Recall that, under the following conditions 
\begin{equation} \label{cond-irred-aper}
r_0 < 1, \qquad \forall i \ge 1, \quad 0< q_i,p_i <1, \quad  S:=1+\sum_{i=1}^{\infty} \prod_{j=1}^i \frac{q_{j-1}}{p_j} < \infty,
\end{equation}
$P$ is irreducible, aperiodic and $\pi$ (unique) is given by: 
$\pi(0)= 1/S, \ \pi(i) = (\prod_{j=1}^i \frac{q_{j-1}}{p_j})/S$. 
Moreover it is well-known that $P$ is reversible w.r.t.~$\pi$. Finally Condition~(\ref{ass-beta}) writes as: 
\begin{equation} \label{cond-D-F-rever}
\alpha_0:= \limsup_i \sqrt{p_iq_{i-1}} + \limsup_i r_i  + \limsup_i \sqrt{q_{i}p_{i+1}}  < 1.
\end{equation} 
Consequently, under Conditions~(\ref{cond-irred-aper}) and (\ref{cond-D-F-rever}), $P$ satisfies (\ref{ineg-gap-gene}) and $r_{ess}(P_{|\ell^2(\pi)})\leq \alpha_0$. In particular, if the sequences $(p_i)_{i\in\N^*}$, $(r_i)_{i\in\N}$ and $(q_i)_{i\in\N}$ in (\ref{matrix-BDMC}) admit a limit when $i\r+\infty$, say $p,r,q$,  then (\ref{ineg-gap-gene}) holds provided that $p>q$. Moreover $r_{ess}(P_{|\ell^2(\pi)})\leq 1 - (\sqrt{p}-\sqrt{q})^2$. 
%
\begin{ex} [State-independent BDMC] \label{ex-BDMC-indep} $\ $ \\[0.12cm]
Let $P$ given by (\ref{matrix-BDMC}) such that, for any $i\geq 1$, $p_i:=p$, $r_i:=r$, $q_i:=q$, with $p,q,r\in[0,1]$ such that $p+r+q=1$ and $p>q>0$. Let $r_0\in(0,1)$ and $\beta_0 := 1-q-\sqrt{pq}$. The bounds for $\varrho_{V}$ with $V(n):= (p/q)^{n/2}$ can be derived from \cite[Prop.~4.1]{HerLedJAP14}, so that (Corollary~\ref{rem-surprise}):   
\leftmargini 1.5em   
\begin{itemize}
	\item if $r_0\in[\beta_0,1)$, then $\varrho_2 \leq r+2\sqrt{pq}$; 
  
	
	\item if $r_0\in(0,\beta_0]$, then : 
\begin{enumerate}[(a)]
	\item in case $\, 2p \leq \big(1-q+\sqrt{pq}\big)^2$: $\varrho_2 \leq r+2\sqrt{pq}$; 
  \item in case $\, 2p > \big(1-q+\sqrt{pq}\big)^2$, setting $\beta_1 := p - \sqrt{pq}  - \sqrt{r\big(r+2\sqrt{pq}\big)}$:
\begin{subequations}
\begin{eqnarray}
& & \varrho_2 = \big|r_0 + \frac{p(1-r_0)}{r_0-1+q}\big| \ \ \text{ when } r_0\in(0,\beta_1] \label{a0-1-L2} \\
& & \varrho_2 \leq  r+2\sqrt{pq}  \ \, \quad\qquad\text{ when } r_0\in[\beta_1,\beta_0). \label{a0-3-L2}
\end{eqnarray}
\end{subequations}
\end{enumerate}
\end{itemize}
\end{ex}

\newcommand{\gc}{\widehat\gamma}
\newcommand{\rhoc}{\widehat\rho}

\begin{rem}[Discussion on the $\ell^2(\pi)$-spectral gap and the decay parameter] \label{l2-spectral-gap} ~\newline
Let $P$ be a BDMC satisfying (\ref{cond-irred-aper}). It can be proved that the decay parameter of $P$, denoted by $\gamma$ in \cite{DooSch95} but by $\gamma_{DS}$ here to avoid confusion, equals to $\varrho_2$, that is (from reversibility):  $\gamma_{DS} =\varrho_2 = \|P-\Pi\|_2$. But note that  $\gamma_{DS}$ is only known for specific instances of BDMC from \cite{DooSch95} (see \cite{Kov10} for a recent contribution). 
For a general Markov kernel $P$, we only have (see also \cite{Pop77,Isa79})
$\gamma_{DS} \leq \varrho_2$. 
In particular, the decay parameter does not provide information on non-reversible RWs with i.d. bounded increments of Section~\ref{sec-stab-expo}.
\end{rem}

\subsection{The Metropolis-Hastings Algorithm} \label{sec-MCMC}
Let $\pi=(\pi(i))_{i\in\N}$ (target distribution) be a probability measure on $\N$ known up to a multiplicative constant. Let $Q:=(Q(i,j))_{(i,j)\in\N^2}$ (proposal kernel) be any transition kernel on $\N$. The associated  Metropolis-Hastings (M-H) Markov kernel $P:=(P(i,j))_{(i,j)\in\N^2}$ is defined by 


$$
P(i,j) := \left \{
    \begin{array}{ll}
       \min\left(Q(i,j)\, ,\, \frac{\pi(j)\, Q(j,i)}{\pi(i)}\right) \quad \quad  \text{if } i\neq j \\[0.15cm]
      1 - \sum_{\ell\neq i} P(i,\ell) \quad\quad\quad \quad \quad  \ \ \text{if } i =j. 
    \end{array}
    \right.
		$$
It is well-known that $P$ is reversible with respect to $\pi$ and that $\pi$ is $P$-invariant.
%
\begin{cor} \label{cor-MCMC}
Assume that $\pi(i)>0$ for every $i\in\N$ and that $\pi$ satisfies \emph{(\ref{pi-tail})} with $\tau\in(0,1)$. Assume that $Q$ is an aperiodic and irreducible Markov kernel on $\N$ such that for every $(i,j)\in \N^2$, $Q(i,j)=0 \Leftrightarrow Q(j,i)=0$, satisfying \emph{(\ref{ass-voisin})}  and the following condition (see \emph{(\ref{cond-lim-intro})})
\begin{equation} \label{cond-lim-MCMC}
\forall {m}=-N,\ldots,N,\quad q_{m} := \lim_{i\r+\infty} Q(i,i+{m}).
\end{equation}
Finally assume that $(q_k,q_{-k}) \neq (0,0)$ for some $k\in\{1,...,N\}$.
Then the associated M-H kernel $P$ satisfies \emph{(\ref{ineg-gap-gene})} and $r_{ess}(P_{|\ell^2(\pi)})\leq  \alpha_0$ with 
\begin{equation} \label{cond-lim-rev-ter}
\alpha_0 :=  1 - \sum_{m=1}^N\big(\sqrt{p_m} - \sqrt{p_{-m}}\big)^2 \quad \text{where}\ \ 
p_k := \begin{cases}
        \min\left(q_k\, ,\, \tau^k\, q_{-k}\right) &  \text{if } k \neq 0 \\[0.15cm]
      1 - \sum_{\ell=1}^N \big(p_\ell + p_{-\ell}\big) & \text{if } k=0. 
    \end{cases}
\end{equation}
If \emph{(\ref{pi-tail})} holds with $\tau=0$, then $p_{m}=0$ for every ${m}=1,\ldots,N$, and the above conclusions holds true with $\alpha_0 :=  p_0$ provided that $p_0<1$. 
\end{cor}
%
\begin{proof}{}
It is well-known that $P$ is irreducible and aperiodic under the basic assumptions on $Q$.  If $Q$ satisfies (\ref{ass-voisin}) for some $N$, then so is $P$ (with the same $N$). Assumption~(\ref{cond-lim-intro}) holds for $P$: $\lim_{i\r+\infty} P(i,i+{m}) = p_m$ with $p_m$ defined in (\ref{cond-lim-rev-ter}). Then apply Corollary~\ref{rem-surprise}. 
\end{proof}
%
\begin{ex} \label{ex-MCMC-q}
Assume that $\pi$ (possibly known up to a multiplicative constant) is such that $\pi(i)>0$ for every $i\in\N$ and satisfies \emph{(\ref{pi-tail})}. Let $Q$ be a transition kernel on $\N$ satisfying 
$$Q(0,0):=r<1,\ Q(0,1):=1-r, \quad\forall i\geq 1,\quad Q(i,i-1) = q,\quad  Q(i,i) = 1-2q,\quad Q(i,i+1)=q$$
for some $q\in(0,1/2]$. 
The associated M-H Markov kernel $P^{(q)}$ is given by $P^{(q)}(0,1) =\min(1-r\, ,\, q\, \pi(1)/\pi(0))$ and 
\begin{gather*}
\forall i\geq 1, \quad P^{(q)}(i,i-1) = q \min\left(1\, ,\, \frac{\pi(i-1)}{\pi(i)}\right) \quad P^{(q)}(i,i+1) = q \min\left(1\, ,\, \frac{\pi(i+1)}{\pi(i)}\right) \\ P^{(q)}(i,i):=1- \sum_{\ell\neq i} P^{(q)}(i,\ell). 
\end{gather*}
The conditions of Corollary~\ref{cor-MCMC} are trivially satisfied. 
 Then $P^{(q)}$ 
satisfies \emph{(\ref{ineg-gap-gene})}. Next $\alpha_0 \equiv \alpha_0(q)$ in (\ref{cond-lim-rev-ter}) is 


\begin{equation}\label{MCMC_alpha_0}
\alpha_0(q) = 1 - q\big(1 - \sqrt{\tau}\big)^2
\end{equation}
since the $p_m$'s in (\ref{cond-lim-rev-ter}) are given by $p_{-1} = q,\ p_0 = 1-q-q\tau,\ p_1 = q\tau$. When $q\in(0,1/2]$, $\alpha_0(q)$ is minimal for $q=1/2$, thus $q=1/2$ provides the minimal 
bound for $r_{ess}(P^{(q)}_{|\ell^2(\pi)})$
The relevant question is to find $q\in(0,1/2]$ providing the minimal value of $\varrho_2\equiv \varrho_2(q)$ (See Example~\ref{ex-MCMC-q-contin}). 
\end{ex}
%


\begin{ex} [Simulation of Poisson distribution with parameter $1$] \label{ex-simu-poisson}
Let $\pi$ be the Poisson distribution with parameter $\lambda :=1$, defined by $\pi(i):=\exp(-1)/i!$. Then \emph{(\ref{pi-tail})} holds with $\tau =0$. Introduce the proposal kernel $Q$ of Example~\ref{ex-MCMC-q}  with $r:=1/2$ and $q :=1/2$. 
The associated M-H kernel $P$ is given by $P^{(q)}(0,0)=P^{(q)}(0,1)=1/2$ and 
\begin{gather*}
\forall i \geq 1, \ P^{(q)}(i,i-1)=\frac{1}{2} \quad P^{(q)}(i,i)=\frac{i}{2(i+1)}, \quad P^{(q)}(i,i+1)=\frac{1}{2(i+1)}.
\end{gather*}
We know from Example~\ref{ex-MCMC-q} that $P^{(q)}$ satisfies \emph{(\ref{ineg-gap-gene})}  and  $r_{ess}(P^{(q)}_{|\ell^2(\pi)}) \leq \alpha_0 = 1/2$.
The rate of convergence $\varrho_2\equiv \varrho_2(q)$ of $P^{(q)}$ is studied in Example~\ref{ex-simu-poisson-cont}. 
\end{ex}
%

%
\section{Bound for $\varrho_2$ via truncation and numerical applications} \label{cas-gap-connu}
Let us consider the following $k$-th truncated (and augmented) matrix $P_k$ associated with $P$: 	
\[ 
\forall (i,j)\in \{0,\ldots,k-1\}^2, \quad P_k(i,j) := \begin{cases}
P(i,j) & \text{ if $\ 0\leq i\leq k-1\ $ and $\ 0\leq j \leq k-2$} \\
\sum_{\ell\ge k-1}P(i,\ell) & \text{ if $\ 0\leq i\leq k-1\ $ and $\ j=k-1$}.
\end{cases}
\]
Let $\sigma(P_k)$ denote the set of eigenvalues of $ P_k$, and define 
$\rho_k := \max\big\{|\lambda|,\, \lambda\in\sigma(P_k),\, |\lambda| <1\big\}.$

Recall that $V(i) := {\pi(i)}^{-1/2}$ and that $\varrho_V$ is defined in (\ref{def-rho-V}). The statement below follows from Proposition~\ref{pro-RW-SG} and from the weak perturbation method in \cite{HerLed14a} applied to $P_{|{\cal B}_{V}}$, for which the drift inequality (\ref{drift-V}) plays an important role. 
%
\begin{pro} \label{th-first}
If $P$ satisfies \emph{(\ref{cond-lim-intro})}, \emph{(\ref{pi-tail})} and \emph{(\ref{MoySautBorne})}, then the following properties holds with $\alpha_0$ given in (\ref{fle-alpha}): 
\begin{enumerate}[(a)]
	\item $\varrho_2 \leq \alpha_0\ \Longleftrightarrow \varrho_{V} \leq \alpha_0$, and in this case we have $\limsup_k\rho_k \leq \alpha_0$; 
	\item $\varrho_2 > \alpha_0\ \Longleftrightarrow \varrho_{V} > \alpha_0$, and in this case we have $\varrho_2 = \varrho_{V} = \lim_k\rho_k$. 
\end{enumerate}
\end{pro}

Below the estimation of the convergence rate $\varrho_2$ for some Metropolis-Hastings Markov kernel $P$ is derived from Proposition~\ref{th-first}. Recall that Inequality (\ref{ineg-gap-intro-rho2}) applies when $P$ is reversible. The generic procedure for the following instances of Markov kernel $P$ is as follows:
\begin{enumerate}
	\item Compute $\alpha_0$ given in (\ref{cond-lim-rev-ter}) and choose a small $\varepsilon>0$
	\item $k:=2$
		\item \label{3}Consider the $k$-order truncated matrix  $P_k$ of the kernel $P$. 
		\item Compute the second highest eigenvalue $\rho_k$ of $P_k$.
		\item If $|\rho_k-\rho_{k-1}|> \varepsilon$ then ($k:=k+1$, return to step \ref{3})
		
		else if $\rho_k > \alpha_0$ then  $\varrho_2 \simeq \rho_k$
		else $\varrho_2 \leq \alpha_0$.
\end{enumerate}
It is clear that the control of the stabilization of the sequence $(\rho_k)_{k \geq 2}$ through the comparison between $|\rho_k-\rho_{k-1}|$ and $\varepsilon$ only provides an estimation of $\varrho_2$. 
\begin{ex} [Example~\ref{ex-MCMC-q} continued] \label{ex-MCMC-q-contin} $\ $ \\
Let us consider the probability distribution $\pi$ given by $\pi(i) := C \, (i+1)\,\tau^i$ for $n\in\N$ where $C$ is a (possibly unknown) normalisation constant and $0<\tau <1$. Then \emph{(\ref{pi-tail})} is satisfied. If we choose an RW as in Example~\ref{ex-MCMC-q} for the proposal kernel, the  associated M-H kernel $P^{(q)}$ is defined by $P^{(q)}(0,1) =\min\left(1-r\,, 2\, q\, \tau\right)$ and 
\begin{gather*}
\forall i\geq 1, \quad P^{(q)}(i,i-1) = q \min\left(1\, ,\, \frac{1}{\tau} \, \frac{i}{i+1}\right) \quad P^{(q)}(i,i+1) = q \min\left(1\, ,\,\tau\, \frac{i+2}{i+1}\right)   \\
    P^{(q)}(i,i):=1- \sum_{\ell\neq i} P^{(q)}(i,\ell). 
\end{gather*}
For $q\in(0,1/2]$, $P^{(q)}$  satisfies \emph{(\ref{ineg-gap-gene})}  with  $r_{ess}(P^{(q)}_{|\ell^2(\pi)}) \leq \alpha_0(q) = 1 - q\big(1 - \sqrt{\tau}\big)^2$ (see (\ref{MCMC_alpha_0})). Table~\ref{Table_tau=3/5} based on Proposition~\ref{th-first} gives the estimate of $\varrho_2(q)$ of $P^{(q)}$. 

\begin{table}
\begin{center}
\setlength{\extrarowheight}{2pt}
\begin{tabular}{|c||c|c|c||c|c|c|}\cline{2-7}
\multicolumn{1}{c|}{} & \multicolumn{3}{c||}{$\boldsymbol{\tau=0.2}$} & \multicolumn{3}{c|}{$\boldsymbol{\tau=0.5}$} \\\hline
 $\boldsymbol{q}$ &  $\boldsymbol{\alpha_0(q,\tau)}$ & $\boldsymbol{\rho_{k}(q)}$  & $\boldsymbol{\varrho_2(q)}$ & $\boldsymbol{\alpha_0(q,\tau)}$ & $\boldsymbol{\rho_{k}(q)}$  & $\boldsymbol{\varrho_2(q)}$ \\\hline\hline
0.1 &  0.9694 & $\rho_{27}\simeq 0.9710 $   & $\simeq 0.9710 $& 0.9914 &  $\rho_{39}\simeq 0.9921$ & $\simeq 0.9921$  \\\hline
0.2  & 0.9389 & $\rho_{30}\simeq 0.9421 $   &  $\simeq 0.9421 $& 0.9828 &  $\rho_{44}\simeq 0.9842$ & $\simeq 0.9842$ \\\hline
0.3 &  0.9083 & $\rho_{31}\simeq 0.9131 $  & $\simeq 0.9131$& 0.9743 & $\rho_{47}\simeq 0.9763$ & $\simeq 0.9763$ \\\hline
0.4 & 0.8778 & $\rho_{32}\simeq 0.8842 $ & $\simeq 0.8842$ & 0.9657 & $\rho_{50} \simeq 0.9684$ & $\simeq 0.9684$  \\\hline
0.5  & 0.8472 & $\rho_{33}\simeq 0.8552$ & $\simeq 0.8552$ & 0.9571 & $\rho_{51}\simeq  0.9605$ & $\simeq 0.9605$ \\\hline
\end{tabular}
\begin{tabular}{|c||c|c|c||c|c|c|}\cline{2-7}
\multicolumn{1}{c|}{} & \multicolumn{3}{c|}{$\boldsymbol{\tau=0.6}$}  & \multicolumn{3}{c|}{$\boldsymbol{\tau=0.8}$}  \\\hline
 $\boldsymbol{q}$ &  $\boldsymbol{\alpha_0(q,\tau)}$ & $\boldsymbol{\rho_{k}(q)}$  & $\boldsymbol{\varrho_2(q)}$ & $\boldsymbol{\alpha_0(q,\tau)}$ & $\boldsymbol{\rho_{k}(q)}$  & $\boldsymbol{\varrho_2(q)}$\\\hline\hline
0.1 & 0.9949 &  $\rho_{44}\simeq 0.9953$ & $\simeq 0.9953$ & 0.99889 & $\rho_{55} \simeq 0.99883$& $\leq 0.99889$ \\\hline
0.2 & 0.9898 &  $\rho_{51}\simeq 0.9906$ & $\simeq 0.9906$ & 0.99777 & $\rho_{66}\simeq  0.99781$ & $\simeq  0.99781$ \\\hline
0.3 & 0.9848 & $\rho_{55}\simeq 0.9860$ & $\simeq 0.9860$ & 0.99666 & $\rho_{73}\simeq 0.9968$ & $\simeq 0.9968$\\\hline
0.4 & 0.9797 & $\rho_{58} \simeq 0.9814$ & $\simeq 0.9814$ & 0.99554 & $\rho_{79} \simeq 0.99579$ & $\simeq 0.99579$\\\hline
0.5 & 0.9746 & $\rho_{60}\simeq  0.9767$ & $\simeq 0.9767$ & 0.99443 & $\rho_{83} \simeq 0.9948$ & $\simeq 0.9948$\\\hline
\end{tabular}
\caption{Results for different values of $\tau$ with $\varepsilon=10^{-5}$. The second eigenvalue $\rho_{k}\equiv\rho_{k}(q)$ of $P_{k}\equiv {P^{(q)}}_{k}$ is obtained from the observed empirical stabilization of $\rho_k$ with respect to $k$.}
\label{Table_tau=3/5}
\end{center}
\end{table}
\end{ex}


\begin{ex} [Example~\ref{ex-simu-poisson} continued] \label{ex-simu-poisson-cont}
Table~\ref{Table_Poisson} based on Proposition~\ref{th-first} gives the estimation of $\varrho_2(q)$ of the M-H $P^{(q)}$ used in the simulation of the Poisson distribution of Example~\ref{ex-simu-poisson}. Note that $q:=1/2$ gives the smallest value of $r_{ess}(P^{(q)}_{|\ell^2(\pi)})=\alpha_0(q)=0.5$, with $\alpha_0$ given by (\ref{MCMC_alpha_0}). However $q:=1/2$ does not provide the minimal rate of convergence in $\ell^2(\pi)$-norm (or in $\cB_{V}$-norm). More precisely,  for $q=1/2$, the kernel $P^{(q)}$ admits some eigenvalues in the annulus $\Gamma:=\{\lambda\in\R : 0.5 < |\lambda| < 1\}$, among which $\varrho_2(q)\approx 0.8090$ is the larger one in absolute value. Actually the minimal rate of convergence  is achieved at $q\approx 0.38$ and note that every value $0.2\leq q< 0.5$ in Table~\ref{Table_Poisson} provides a minimal rate than for $q:=1/2$. It could be conjectured from numerical evidence that for $q\le q_0$ with $q_0\approx 0.35$, $\varrho_2=\alpha_0(q)$. 
\begin{table}
\begin{center}
\setlength{\extrarowheight}{2pt}
\begin{tabular}{|c|c|c|c|}\hline
$\boldsymbol{q}$ &  $\boldsymbol{\alpha_0(q)\equiv r_{ess}(P^{(q)})}$ & $\boldsymbol{\rho_{k}(q)}$  & $\boldsymbol{\varrho_2(q)}$ \\\hline\hline
0.1 & 0.9 &  $\rho_{37} \simeq 0.9003$ & $\simeq 0.9003$ \\\hline
0.2 & 0.8 &  $\rho_{83} \simeq 0.8008$ & $\simeq 0.8008$ \\\hline
0.3 & 0.7 & $\rho_{151} \simeq 0.7015$ & $\simeq 0.7015$\\\hline
0.38 & 0.62 & $\rho_{61} \simeq 0.6301$ & $\simeq 0.6301$ \\\hline
0.4 & 0.6 & $\rho_{17} \simeq 0.6568$ & $\simeq 0.6568$ \\\hline
0.5 & 0.5 & $\rho_{14} \simeq 0.8090$ & $\simeq 0.8090$ \\\hline
\end{tabular}
\caption{$\rho_{k}\equiv\rho_k(q)$ is obtained from the observed empirical stabilization of $\rho_k$ with $\varepsilon=10^{-5}$. }
\label{Table_Poisson}
\end{center}
\end{table}
\end{ex}

\newpage

\bibliographystyle{alpha}

\end{document}